\documentclass{amsart}

\usepackage{bbm}

\usepackage{amssymb}
\usepackage{enumerate}
\usepackage{mathrsfs}
\usepackage{hyperref}
\usepackage[justification=centering, skip=20pt]{caption}
\usepackage{tikz}
\usepackage[all,cmtip]{xy}
\usepackage{graphicx}

\usepackage{setspace}

\hypersetup{
    colorlinks,
    citecolor=black,
    filecolor=black,
    linkcolor=black,
    urlcolor=black
    }

\usepackage{geometry}    
\newtheorem{thm}{Theorem}
\newtheorem{cor}[thm]{Corollary}
\newtheorem{prop}[thm]{Proposition}
\newtheorem{lem}[thm]{Lemma}

\newtheorem{defn}[thm]{Definition}

\newtheorem{rem}[thm]{Remark}

\newcommand{\R}[0]{\mathbb{R}}
\newcommand{\Z}[0]{\mathbb{Z}}
\newcommand{\N}[0]{\mathbb{N}}
\newcommand{\C}[0]{\mathbb{C}}
\renewcommand{\t}[1]{\textup{#1}}

\renewcommand{\l}[0]{\langle}
\renewcommand{\r}[0]{\rangle}

\numberwithin{equation}{section}

\newcommand\numberthis{\addtocounter{equation}{1}\tag{\theequation}}

\usepackage{etoolbox}

\makeatletter
\patchcmd{\@settitle}{\uppercasenonmath\@title}{}{}{}
\patchcmd{\@setauthors}{\MakeUppercase}{}{}{}
\patchcmd{\section}{\scshape}{}{}{}
\makeatother

\title[Strichartz estimates on products of odd-dimensional spheres]{Strichartz estimates for the Schr\"odinger equation on products of odd-dimensional spheres}

\author[Y. Zhang]{Yunfeng Zhang}
\email{yunfengzhang108@gmail.com}

\begin{document}

\onehalfspacing

\begin{abstract}
We prove Strichartz estimates for the Schr\"odinger equation which are scale-invariant up to an $\varepsilon$-loss on products of odd-dimensional spheres. Namely, for any product of odd-dimensional spheres $M=\mathbb{S}^{d_1}\times\cdots\times\mathbb{S}^{d_r}$ (so that $M$ is of dimension $d=d_1+\cdots+d_r$ and rank $r$) equipped with rational metrics, the following Strichartz estimate 
\begin{equation*}
\|e^{it\Delta}f\|_{L^p(I\times M)}\leq C_\varepsilon\|f\|_{H^{\frac{d}{2}-\frac{d+2}{p}+\varepsilon}(M)}
\end{equation*}
holds for any $p\geq 2+\frac{8(s-1)}{sr}$, 
where 
$$s=\max\left\{\frac{2d_i}{d_i-1}, i=1,\ldots,r\right\}.$$
\end{abstract}

\maketitle


\section{Introduction}
Let $M$ be a compact Riemannian manifold and let $\Delta$ denote the Laplace-Beltrami operator. Let $e_1,e_2,\ldots$ denote an orthonormal basis of $L^2(M)$ consisting of eigenfunctions of $-\Delta$ with respect to the eigenvalues $0=\lambda_1\leq\lambda_2\leq\ldots$, and define the Sobolev spaces
$$H^s(M):=\left\{f\in L^2(M):\  f=\sum_{i}f_ie_i,\ \|f\|_{H^s(M)}:=\left(\sum_i|f_i|^2(\lambda_i^{s}+1)\right)^{1/2}<\infty\right\}.$$
We may define the one-parameter unitary group $e^{it\Delta}: H^s(M)\to H^s(M)$ ($t\in\R$), such that 
\begin{align*}
e^{it\Delta}f=\sum_{i}e^{-it\lambda_i}f_ie_i, \t{ for }f=\sum_if_ie_i. 
\end{align*} 
Then $u(t,x)=e^{it\Delta}f(x)$ provides the solution to the linear Schr\"odinger equation 
$$\left\{\begin{array}{ll}
i\partial_t u(t,x)+\Delta u(t,x)=0,
\\
u(0,x)=f(x).
\end{array}\right.$$
By Sobolev embedding, $H^s(M)\subset L^{\frac{2d}{d-2s}}(M)$ for $s<\frac{d}{2}$. An interesting behavior of the Schr\"odinger flow $e^{it\Delta}$ is that although for each fixed $t\in\R$ and $f\in H^s(M)$, $e^{it\Delta}f$ may not lie in $L^{q}(M)$ for any $q>\frac{2d}{d-2s}$, but when averaging over $t$ in any (finite) interval $I$, we expect estimates of the form 
\begin{equation}\label{Strichartz}
\|e^{it\Delta}f\|_{L^p(I,L^q(M))}\leq C \|f\|_{H^s(M)}
\end{equation}
for some $q>\frac{2d}{d-2s}$, thus the solution $e^{it\Delta}f$ actually gains integrability almost surely in time. We refer to such estimates as Strichartz estimates, which characterize the dispersive nature of solutions and have important applications in solving nonlinear Schr\"odinger equations. If the underlying manifold $M$ is noncompact and satisfies some geometric restraints such as nontrapping for geodesics, these estimates are usually proved by first establishing $L^\infty(M)$ decay of solutions with $L^1(M)$ initial data as $t\to\infty$, making use of the assumptions on $M$ so that the solutions would genuinely  ``disperse to infinity". For reference, see the original work \cite{GV95,KT98} on the Euclidean spaces, and \cite{Ban07,Pie08,AP09,IS09,APV11,FMM15} on other noncompact manifolds such as hyperbolic spaces, Damek-Ricci spaces, and some locally symmetric spaces. However, in the compact picture, naively, there is no ``infinity'' for the solutions to disperse to, so the global geometry of $M$ would become more relevant, and proof of the Strichartz estimates would somehow combine global geometry with the local dispersion of solutions in a nice way. 

By a scale consideration, a necessary condition for the above Strichartz estimates to hold is
\begin{equation}\label{scaling}
s\geq \frac{d}{2}-\frac{2}{p}-\frac{d}{q},
\end{equation} 
for $p,q\geq 2$. Here $d$ denotes the dimension of $M$. If the above equality holds, we say the Strichartz estimates are scale-invariant, and non-scale-invariant otherwise. A fundamental result from \cite{BGT04} establishes that on a general compact Riemannian manifold 
\eqref{Strichartz} holds 
for all admissible pairs $(p,q)$
\begin{align*}
\frac{d}{2}=\frac{2}{p}+\frac{d}{q}, \ p,q\geq 2, \ (p,q,d)\neq (2,\infty,2),
\end{align*} 
with $s=\frac{1}{p}$. 
These estimates are non-scale-invariant, and are sharp for $(p,q,s)=(2,\frac{2d}{d-2}, \frac{1}{2})$ on spheres of dimension $d\geq 3$. The proof is in its essence a local argument, which establishes the same kind of $L^\infty$ decay as in the noncompact setting for solutions with frequency localized initial data, but only for a short period of time due to the compact nature of $M$, which results in the non-scale-invariance of the estimates. See also \cite{CE20} for some sharp non-scale-invariant estimates on spheres. 
On the other hand, known results of scale-invariant estimates (and their $\varepsilon$-loss versions) are all established by truly global arguments. Consider the special cases when $p=q$, then the scale-invariant estimates read
\begin{equation}\label{StrCri}
\|e^{it\Delta}f\|_{L^p(I\times M)}\leq C\|f\|_{H^{\frac{d}{2}-\frac{d+2}{p}}(M)}.
\end{equation}
We summarize known results of the above type: 
\begin{itemize}
\item In \cite{BGT07,Her13}, $\eqref{StrCri}$ is established for $p>4$ for any sphere $\mathbb{S}^d$ equipped with standard metrics and generally for any Zoll manifolds, of dimension $d\geq 3$, and $p\geq 6$ for the two sphere or any Zoll surface; $p>4$ is also the optimal range for $\mathbb{S}^d$ ($d\geq 3$) (\cite{BGT04}). The global geometry of the spheres and the Zoll manifolds is explored in the proof in terms of the explicit spectral distribution of the Laplacian. 
\item In \cite{Bou93,Bou13,BD15,KV16}, $\eqref{StrCri}$ is established for the optimal range $p>\frac{2(d+2)}{d}$ for rectangular tori, with the same estimates but with an $\varepsilon$-loss established for nonrectangular tori. The proofs use up-to-date tools of harmonic analysis on tori. Note also in \cite{Zha20}, it is observed that the $\varepsilon$-loss can be eliminated for rational nonrectangular tori. 

\item In \cite{Zha20}, $\eqref{StrCri}$ is established for $p\geq\frac{2(r+4)}{r}$ for any compact Lie group equipped with  standard metrics. Later in \cite{Zha21}, the same result is established for any compact globally symmetric space. 
Here $r$ is the rank of the underlying space, i.e. the dimension of a maximal totally geodesic submanifold. The proof explicitly applies representation theory and harmonic analysis for groups and symmetric spaces. Also certain upgrade on the range of $p$ is provided in \cite{Zha22} for compact Lie groups. 

\end{itemize}

In this note, we provide the same kind of upgrade on the range of $p$ for a special class of compact globally symmetric spaces as in \cite{Zha22} for compact Lie groups. We establish the following 

\begin{thm}\label{maintheorem}
Let $M=\mathbb{S}^{d_1}\times\cdots\times\mathbb{S}^{d_r}$ be any product of odd dimensional spheres, equipped with a rational metric, of dimension $d=d_1+\cdots+d_r$ and rank $r$.
Let  
$$s=\max\left\{\frac{2d_i}{d_i-1}, i=1,\ldots,r\right\}.$$
Let $\varepsilon>0$.  
Then the following Strichartz estimates hold
\begin{equation}\label{StrCri}
\|e^{it\Delta}f\|_{L^p(I\times M)}\leq C_\varepsilon\|f\|_{H^{\frac{d}{2}-\frac{d+2}{p}+\varepsilon}(M)}.
\end{equation}
for any $p\geq 2+\frac{8(s-1)}{sr}$.  
\end{thm}

As in \cite{Zha20} and \cite{Zha21}, we relate the above Strichartz estimates to some Lie theoretic Weyl-type quadratic exponential sums as follows:
\begin{align*}
K_N(t,[iH])=\sum_{\lambda\in\Lambda^+}\varphi\left(\frac{-|\lambda+\rho|^2+|\rho|^2}{N^2}\right)e^{it(-|\lambda+\rho|^2+|\rho|^2)}d_\lambda\Phi_\lambda([iH]),\ H\in\mathfrak{a}, 
\end{align*} 
where $\Lambda^+$ is the part of the weight lattice $\Lambda\subset\mathfrak{a}^*$ inside of a Weyl chamber associated to a root system, $\varphi$ is any smooth cutoff function, $d_\lambda$ as the dimension of spherical representations is a polynomial function in $\lambda$, and most importantly 
\begin{align*}
\Phi_\lambda([iH])=\sum_{j=1}^q c_je^{i\lambda_j(H)}, \ \lambda_j\in\Lambda,\ c_j\geq 0,\ H\in\mathfrak{a}, 
\end{align*}
as the spherical functions are some special exponential sums. Such $K_{N}(t,[iH])$ appear as the mollified kernel functions in the linear Schr\"odinger flow
\begin{align*}
e^{it\Delta}\varphi(N^{-2}\Delta)f=f*K_{N}(t,\cdot).
\end{align*} 
When the underlying manifold is actually a compact Lie group, the spherical functions $\Phi_\lambda$ are given explicitly by the Weyl character and dimension formulas, which allows us to establish in \cite{Zha21} the following ``major-arc'' (i.e., time intervals centered around rationals) estimates for the above quadratic exponential sums using classical techniques. 

\begin{prop}[\cite{Zha20}]\label{exp20}
Let $\mathbb{S}^1$ stand for the standard circle of unit length, and let $\|\cdot\|$ stand for the distance from 0 on $\mathbb{S}^1$. Define the major arcs
\begin{align*}
\mathcal{M}_{a,q}=\left\{t\in\mathbb{S}^1:\ \left\|t-\frac{a}{q}\right\|<\frac{1}{qN}\right\}
\end{align*}
where
\begin{align*}
a\in\mathbb{Z}_{\geq 0},\ q\in\N,\ a<q,\ (a,q)=1,\ q<N. 
\end{align*}
Let $G_i$ be simple Lie groups of dimension $d_i$ and rank $r_i$, $i=1,\ldots,m$, and let $M=G_1\times\cdots\times G_m$ be their product equipped with a rational metric. Let $d=\sum_i d_i$ and $r=\sum_i r_i$ be the dimension and rank of $M$ respectively. 
Let 
$$s=\max\left\{\frac{2d_i}{d_i-r_i}\right\}.$$
Consider the mollified Schr\"odinger kernel $K_N$. Then there is some $T>0$ such that $K_N(t,x)=K_N(t+T,x)$ for any $t\in\mathbb{R}$ and $x\in M$, and for each $p>s$, it holds that 
\begin{align}
\|K_N(t,\cdot)\|_{L^p(M)}\leq C\frac{N^{d-\frac{d}{p}}}{[\sqrt{q}(1+N\|\frac{t}{T}-\frac{a}{q}\|^{1/2})]^r}
\end{align}
for $\frac{t}{T}\in\mathcal{M}_{a,q}$.
\end{prop}

In this paper, we provide the following analogue of the above result on products of odd-dimensional spheres. 

\begin{thm}\label{keyproposition}
Let $M=\mathbb{S}^{d_1}\times\cdots\times\mathbb{S}^{d_r}$ be any product of odd dimensional spheres, equipped with a rational metric, of dimension $d=d_1+\cdots+d_r$ and rank $r$.
Let  
$$s=\max\left\{\frac{2d_i}{d_i-1}, i=1,\ldots,r\right\}.$$  
Consider the mollified Schr\"odinger kernel $K_N$. Then there is some $T>0$ such that $K_N(t,x)=K_N(t+T,x)$ for any $t\in\mathbb{R}$ and $x\in M$, and for each $p\geq s$, it holds that 
\begin{align}\label{KNtH}
\|K_N(t,\cdot)\|_{L^p(M)}\leq C_\varepsilon\frac{N^{d-\frac{d}{p}+\varepsilon}}{[\sqrt{q}(1+N\|\frac{t}{T}-\frac{a}{q}\|^{1/2})]^r}
\end{align}
for $\frac{t}{T}\in\mathcal{M}_{a,q}$.
\end{thm}

From this key theorem, the derivation of Theorem \ref{maintheorem} is a rather standard procedure that combines the Hardy-Littlewood circle method with harmonic analysis on groups and symmetric spaces; for details see the proof of Theorem 1 in \cite{Zha21} and the proof of Theorem 1.3 in \cite{Zha22}. 
The proof of Theorem \ref{keyproposition} exploits three formulas of the spherical functions $\Phi_\lambda$ on odd-dimensional spheres. The first is an integral formula near the identity element, which will provide the part of the desired estimate \label{KNtH} on $N^{-1}$ neighborhoods of the identity element. The second is the expression of spherical functions as the Jacobi polynomials associated to root systems, which along with the first will provide the desired bound on $N^{-1}$ neighborhoods of all the corners of any maximal torus. These two parts of the estimate \label{KNtH} are also valid on a general compact globally symmetric space. The third are the explicit formulas of ultraspherical polynomials, which express the spherical functions on odd-dimensional spheres. 

We provide the following application of Theorem \ref{maintheorem} to nonlinear Schr\"odinger equations and refer to \cite{Zha21} for a proof. 

\begin{thm}\label{application}
Consider the Cauchy problem for the nonlinear Schr\"odinger equation
\begin{align}\label{Cauchyproblem}
\left\{
\begin{array}{ll}
i\partial_t u+\Delta u=F(u), & \ u=u(t,x),\ t\in\R, \ x\in M,\\
u(0,x)=u_0(x)\in H^s(M), & \ x\in M.
\end{array}
\right.
\end{align}
Suppose $F(u)=F(u,\bar{u})$ is a polynomial function in $u$ and its complex conjugate $\bar{u}$ of degree $\beta$ such that $F(0)=0$. Let $M=\mathbb{S}^{d_1}\times\cdots\times\mathbb{S}^{d_r}$ be any product of odd dimensional spheres, equipped with a rational metric, of dimension $d=d_1+\cdots+d_r$ and rank $r$.
Let  
$s=\max\left\{\frac{2d_i}{d_i-1}, i=1,\ldots,r\right\}$
and 
$p_0=2+\frac{8(s-1)}{sr}$. Then 
\eqref{Cauchyproblem} is uniformly well-posedness in $H^s(M)$ for any $s>\frac{d}{2}-\frac{2}{\max(\beta-1,p_0)}$. In particular, if $\beta\geq 1+p_0$, then \eqref{Cauchyproblem} is uniformly well-posedness in $H^s(M)$ for any $s>s_c:=\frac{d}{2}-\frac{2}{\beta-1}$. 
\end{thm}

\section{Preliminaries}
Let $M=\mathbb{S}^{d_1}\times\cdots\times\mathbb{S}^{d_r}$  be a product of odd-dimensional spheres. We will proceed rather generally and view $M=U/K$ as simply connected symmetric spaces of the compact type, and review their properties (\cite{Hel84}, \cite{Hel01}, \cite{Hel08}, \cite{Tak94}). 
Let the Lie algebras of $U,K$ be $\mathfrak{u},\mathfrak{k}$ respectively, and we consider their dual symmetric pair $\mathfrak{g},\mathfrak{k}$ of Lie algebras of the noncompact type such that both $\mathfrak{u},\mathfrak{g}$ lie in the same complexification $\mathfrak{g}^\mathbb{C}$. 
The negative of the Killing form $-\langle \ ,\  \rangle$ defined on $\mathfrak{u}$ induces the Killing metric on $U/K$ invariant under the action of $U$. Slightly more generally, we consider rational products of these metrics 
$g=\otimes_{j=1}^r \beta_jg_{j}$, where the $g_j$'s are the Killing metrics on the irreducible components $\mathbb{S}^{d_j}$ of $M$, and the quotients among the $\beta_j$'s are rational numbers.

Let $(\delta, V_\delta)$ be an irreducible unitary representation of $U$ and let $V_\delta^K$ be the space of vectors $v\in V_\delta$ fixed under $\delta(K)$. 
We say $\delta$ is spherical if $V_\delta^K\neq 0$. Let $\delta$ be such an irreducible spherical representation of $U$. Then $V_\delta^K$ is spanned by a single unit vector ${\bf e}$, and let 
\begin{align}\label{matrixcoefficients}
H_\delta(U/K)=\{\langle\delta(u){\bf e}, v\rangle_{V_\delta}: v\in V_\delta\}.
\end{align} 
Let $\widehat{U}_K$ be the set of equivalence classes of spherical representations of $U$ with respect to $K$.  
The theory of Peter-Weyl gives the Hilbert space decomposition 
\begin{align*}
L^2(U/K)= \bigoplus_{\delta\in \widehat{U}_K} H_\delta(U/K).
\end{align*}
Here the $L^2$ space is of course defined using the $U$-invariant measure on $U/K$. 
Define the spherical functions
\begin{align*}
\Phi_\delta(u):=\langle\delta(u){\bf e}, {\bf e}\rangle_{V_\delta}\in H_\delta(U/K),
\end{align*}  
then the $L^2$ projections $P_\delta: L^2(U/K)\to H_\delta(U/K)$ can be realized by convolution with $d_\delta\Phi_\delta$, so we  have the $L^2$ spherical Fourier series
\begin{align*}
f=\sum_{\delta\in \widehat{U}_K} d_\delta f *  \Phi_\delta
=\sum_{\delta\in \widehat{U}_K} d_\delta \Phi_\delta * f.
\end{align*} 
Here the convolution on $U/K$ is of course defined by pulling back the functions to $U$ and then applying the group convolution.

Next we explicitly characterizes the Fourier dual. Let $\mathfrak{g}=\mathfrak{k}+\mathfrak{p}$ be the Cartan decomposition and $\mathfrak{a}$ be the maximal abelian subspace of $\mathfrak{p}$. Let $\Sigma\subset\mathfrak{a}^*$ denote the restricted root system and let $m_\lambda\in\N$ ($\lambda\in\Sigma$) denote the  multiplicity function. Let $\mathfrak{b}$ be a maximal abelian subspace of the centralizer $\mathfrak{m}$ of $\mathfrak{a}$ in $\mathfrak{k}$ and let $\mathfrak{h}_\R=\mathfrak{a}+i\mathfrak{b}$. 
Let  
$\Sigma^+$ denote a set of positive restricted roots in $\Sigma$ with respect to an order in $\mathfrak{a}^*$ compatible with one in $\mathfrak{h}_\R^*$. Then we have the Iwasawa decomposition 
\begin{align}\label{Iwasawa}
\mathfrak{g}=\mathfrak{n}+\mathfrak{a}+\mathfrak{k}
\end{align}
where $\mathfrak{n}=\sum_{\lambda\in \Sigma^+}\mathfrak{g}_\lambda$ is the direct sum of positive restricted root spaces $\mathfrak{g}_\lambda$. Let $r$ and $d$ be the rank and dimension of $U/K$ respectively. The dimension of $\mathfrak{g}_\lambda$ being $m_\lambda$, the Iwasawa decomposition implies 
\begin{align}\label{numberofroots}
\sum_{\lambda\in\Sigma^+}m_\lambda=d-r. 
\end{align}
Let 
$
\Sigma_*:=\{\alpha\in\Sigma: 2\alpha\notin\Sigma\}
$ be the root system consisting of inmultiplicable roots. Let the weight lattice $\Lambda$ be
\begin{align}\label{Weightlattice}
\Lambda:=\{\lambda\in\mathfrak{a}^*: \frac{\langle\lambda,\alpha\rangle}{\langle\alpha,\alpha\rangle}\in\mathbb{Z}, \ \text{for all }\alpha\in\Sigma_*\}.
\end{align}
Let $\Gamma$ be the restricted root lattice generated by the root system $2\cdot \Sigma$. Then $\Gamma\subset \Lambda$. Let $\Sigma_*^+=\Sigma^+\cap \Sigma_*$ be the set of positive roots in $\Sigma_*$.
Let 
\begin{align*}
\Lambda^+:=\{\lambda\in\mathfrak{a}^*: \frac{\langle\lambda,\alpha\rangle}{\langle\alpha,\alpha\rangle}\in\mathbb{Z}_{\geq 0}, \ \text{for all }\alpha\in\Sigma_*^+\}
\end{align*}
be the set of dominant weights. Given any irreducible spherical representation of $\delta\in\widehat{U}_K$, the highest weight of $\delta$ vanishes on $\mathfrak{b}$ and restricts on $\mathfrak{a}$ as an element in $\Lambda^+$. This gives the isomorphism 
\begin{align}\label{SphericalWeights}
\Lambda^+\cong\widehat{U}_K.
\end{align}
We can also express $\Lambda, \Lambda^+$ in terms of a basis. Let $\{\alpha_1,\cdots, \alpha_r\}$ be the set of simple roots in $\Sigma_*^+$. Let $\{w_1,\cdots, w_r\}$ be the fundamental weights, the dual basis to the (half) coroot basis $\{\frac{\alpha_1}{\langle\alpha_1, \alpha_1\rangle}, \cdots, \frac{\alpha_r}{\langle\alpha_r, \alpha_r\rangle}\}$. Then  
\begin{align*}
\Lambda&=\mathbb{Z}w_1+\cdots+\mathbb{Z}w_r,\\
\Lambda^+&=\mathbb{Z}_{\geq 0}w_1+\cdots+\mathbb{Z}_{\geq 0}w_r. 
\end{align*}

Consider the map 
$
i\mathfrak{a}\to U/K, \ iH\mapsto \exp(iH) K.
$ 
Let $A$ denote the image of the map, then 
\begin{align*}
A\cong i\mathfrak{a}/\Gamma^{\vee}
\end{align*}
where $\Gamma^{\vee}=\{iH\in i\mathfrak{a}: \exp(iH)\in K\}$ is a lattice of $i\mathfrak{a}$. Then 
\begin{align*}
\Gamma^{\vee}=2\pi i\mathbb{Z}\frac{H_{\alpha_1}}{\langle\alpha_1, \alpha_1\rangle}+ \cdots+ 2\pi i\mathbb{Z}\frac{H_{\alpha_r}}{\langle {\alpha_r}, \alpha_r\rangle}.
\end{align*}
Here $H_{\alpha_i}\in\mathfrak{a}$ corresponds to $\alpha_i\in\mathfrak{a}^*$ via the Killing form on $\mathfrak{a}$. We have the isomorphism between $\Lambda$ and the character group $\widehat{A}$ of $A$
\begin{align*}
\Lambda\xrightarrow{\sim}\widehat{A}, \ \lambda\mapsto e^\lambda. 
\end{align*}
Define the cells in $A$ to be the connected components of $A\setminus\cup_{\alpha\in\Sigma}\{[iH]\in A: \alpha(H)\in\pi\mathbb{Z}\}$, and the hyperplanes in $\{[iH]\in A: \alpha(H)\in\pi \Z\}$ are called cell walls. 
For $H\in\mathfrak{a}$, we say $[iH]\in A$ is a corner if $\alpha(H)\in \pi \mathbb{Z}$ for all $\alpha\in\Sigma$. 
Every corner is fixed under the action of $K$ and the set of corners in $A$ is isomorphic to the finite set $\Lambda^{\vee}/\Gamma^{\vee}$. 

We now specialize to rank one spaces and especially odd-dimensional spheres. 
Let $M=U/K$ be a simply connected compact symmetric space of rank 1. Then the restricted root system $\Sigma$ is either $\{\pm \alpha\}$ or $\{\pm \frac{\alpha}{2},\pm\alpha\}$. In both cases, the weight lattice is $\Lambda=\mathbb{Z}\alpha$. Let $A=\mathbb{R}/2\pi\mathbb{Z}$  be the maximal torus, then $e^{n\alpha}=e^{in\theta}$, $\theta\in A$. The two cells of $A$ are $(0,\pi)$ and $(\pi,2\pi)$, with $0,\pi$ being the two corners. Let $m_\alpha$ and $m_{\frac{\alpha}{2}}$ be respectively the multiplicity of $\alpha$ and $\frac{\alpha}{2}$; if the restricted root system is $\{\pm \alpha\}$, then let $m_{\frac{\alpha}{2}}=0$.  Then for $n\in\mathbb{Z}_{\geq 0}\cong \mathbb{Z}_{\geq 0}\alpha\cong\Lambda^+$, the spherical function $\Phi_n$ restricted on $A$ is (see Theorem 4.5 of Chapter V in \cite{Hel84})
\begin{align*}
\Phi_n=\binom{n+a}{n}^{-1}P_n^{(a,b)}(\cos\theta),
\end{align*}
where $\{P_n^{(a,b)}: n\in\mathbb{Z}_{\geq 0}\}$ is the set of Jacobi polynomials (see \cite{Sze75}) with parameters  
\begin{align*}
a=\frac{1}{2}(m_{\frac{\alpha}{2}}+m_\alpha-1), \ b=\frac{1}{2}(m_\alpha-1). 
\end{align*}
The cases when $m_{\frac{\alpha}{2}}=0$ correspond to spheres of dimension $d=m_\alpha+1$, and the Jacobi polynomials in these cases are usually called ultraspherical polynomials. If $d$ is odd, we have explicit formulas for these polynomials. Let $\{\Phi_n^{(\lambda)}, n\in\mathbb{Z}_{\geq 0}\}$ denote the spherical functions on the $(2\lambda+1)$-dimensional sphere, $\lambda\in\mathbb{N}$, then (see Equation (4.7.3) and (8.4.13) in \cite{Sze75})
\begin{align*}
\Phi_n^{(\lambda)}(\theta)=2\binom{n+2\lambda-1}{n}^{-1}\alpha_n\sum_{\nu=0}^{\lambda-1}\alpha_\nu&\frac{(1-\lambda)\cdots(\nu-\lambda)}{(n+\lambda-1)\cdots(n+\lambda-\nu)}\\
&\cdot\frac{\cos((n-\nu+\lambda)\theta-(\nu+\lambda)\pi/2)}{(2\sin\theta)^{\nu+\lambda}}\numberthis\label{sphericalsphere}
\end{align*}
where $\alpha_n:=\binom{n+\lambda-1}{n}$.

We now review the functional calculus of the Laplace-Beltrami operator.
Let $\lambda\in\Lambda^+\cong \widehat{U}_K$ and $H_\lambda(U/K)$ be the space of matrix coefficients associated to $\lambda$ as in \eqref{matrixcoefficients}. 
For any $f\in H_\lambda(U/K)$, we have 
\begin{align}\label{Eigenvalues}
\Delta f=(-\langle\lambda+\rho, \lambda+\rho\rangle+\langle\rho, \rho\rangle)\cdot f,
\end{align} 
where 
\begin{align}\label{rho}
\rho=\frac{1}{2}\sum_{\alpha\in\Sigma^+}m_\alpha \alpha.
\end{align}

Let $f\in L^2(U/K)$ and consider the spherical Fourier series  
$f=\sum_{\lambda\in\Lambda^+} d_\lambda f *  \Phi_\lambda$.  
Then for any bounded Borel function $F:\mathbb{R}\to\mathbb{C}$, we have  
\begin{equation*}
F(\Delta)f=\sum_{\lambda\in\Lambda^+}F(-|\lambda+\rho|^2+|\rho|^2) d_\lambda f *  \Phi_\lambda. 
\end{equation*}
In particular, we have 
\begin{equation}\label{SchFlo}
e^{it\Delta}f=\sum_{\lambda\in\Lambda^+}e^{it(-|\lambda+\rho|^2+|\rho|^2)} d_\lambda f *  \Phi_\lambda.
\end{equation}
Define $P_Nf=\varphi(N^{-2}\Delta)f$, then 
\begin{equation}\label{SchKer}
P_Ne^{it\Delta}f=\sum_{\lambda\in\Lambda^+}\varphi\left(\frac{-|\lambda+\rho|^2+|\rho|^2}{N^2}\right)e^{it(-|\lambda+\rho|^2+|\rho|^2)} d_\lambda f *  \Phi_\lambda.
\end{equation}
In particular, let 
\begin{align}\label{KernelComponent}
K_N(t,x)=\sum_{\lambda\in\Lambda^+}\varphi\left(\frac{-|\lambda+\rho|^2+|\rho|^2}{N^2}\right)e^{it(-|\lambda+\rho|^2+|\rho|^2)} d_\lambda \Phi_\lambda,
\end{align}
then we have 
\begin{align}\label{OperatorKernel}
P_Ne^{it\Delta}f=f*K_N(t,\cdot)=K_N(t,\cdot)*f. 
\end{align}
We call $K_N(t,x)$ the mollified Schr\"odinger kernel on $U/K$. By  rationality of the weight lattice, there is indeed some $T>0$ such that $K_N(t,x)=K_N(t+T,x)$ for any $t\in\mathbb{R}$ and $x\in U/K$. We will need the following lemma on $d_\lambda$, which is a direct consequence of the Weyl dimension formula.

\begin{lem}\label{dimensionformula}
Let $\Phi\subset\mathfrak{h}^*_\R$ denotes the root system associated to $(\mathfrak{g}^\C,\mathfrak{h})$ which restrict on $\mathfrak{a}$ gives $\Sigma\cup\{0\}$. Let $\Phi^+$ be the set of positive roots with respect to the ordering on $\mathfrak{h}^*_\R$ compatible with $\mathfrak{a}^*$. Then the dimension $d_\lambda$ ($\lambda\in\Lambda^+\cong\widehat{U}_K$) equals 
\begin{align}
d_\lambda=\frac{\prod_{\alpha\in\Phi^+, \alpha|_\mathfrak{a}\neq 0}\langle\lambda+\rho', \alpha\rangle}{\prod_{\alpha\in\Phi^+, \alpha|_\mathfrak{a}\neq 0}\langle\rho', \alpha\rangle}, \ \text{ for }\rho'=\frac{1}{2}\sum_{\alpha\in\Phi^+}\alpha. 
\end{align}
As a consequence, $d_\lambda$ as a polynomial function of $\lambda$ is of degree $d-r$. 
\end{lem}

At last, we review the following integral formula that reduces integration on $U/K$ to that on $A$ (see $\S$7 of Chapter II in \cite{Tak94}). 
\begin{lem}
Define 
$$D([iH]):=\prod_{\alpha\in\Sigma^+}|\sin \langle\alpha,H\rangle|^{m_\alpha}.$$
Then for any $K$-invariant continuous function $f$ on $U/K$, we have 
\begin{align}\label{integralformula}
\int_{U/K}f\ d(uK)=c\int_A f([iH])D([iH])\ d([iH]).
\end{align}
Here $c$ is a uniform constant depending on the normalization of 
the invariant measures $d(uK)$ and $d([iH])$ on $U/K$ and $A$ respectively. 
\end{lem}

\section{Dispersive estimates near the corners}
From now on, the notation $a\lesssim b$ stands for $a\leq Cb$ for some positive constant $C$. 
In this section, we prove the following part of Theorem \ref{keyproposition}, which establishes the desired estimates for the mollied Schr\"odinger kernel near the corners on a general compact symmetric space. 

\begin{prop}\label{Lppart1}
Let $U/K$ be any compact globally symmetric space equipped with a rational metric. Consider the associated mollified Schr\"odinger kernel $K_N(t,x)$. Let $T>0$ be such that $K_N(t,x)=K_N(t+T,x)$ for any $t\in\mathbb{R}$ and $x\in U/K$. Pick any $H_0\in\mathfrak{a}$ such that $[iH_0]\in A$ is a corner in a maximal torus of $U/K$. Let 
$$U_{[iH_0],N^{-1}}=\{x\in U/K:\ d(x,[iH_0])\leq N^{-1}\}.$$
Then for any $p>0$, it holds that 
\begin{align}\label{KNtH}
\|K_N(t,\cdot)\|_{L^p\left(U_{[iH_0],N^{-1}}\right)}\leq C_\varepsilon\frac{N^{d-\frac{d}{p}+\varepsilon}}{[\sqrt{q}(1+N\|\frac{t}{T}-\frac{a}{q}\|^{1/2})]^r}
\end{align}
for $\frac{t}{T}\in\mathcal{M}_{a,q}$. 
\end{prop}

By Proposition 9.4 of Ch. III in \cite{Hel08}, the spherical function $\Phi_\lambda$ for $\lambda\in \Lambda^+$ on a maximal torus equals 
\begin{align*}
\Phi_\lambda=\sum_{i=1}^qc_ie^{\lambda_i}, \ \lambda_i\in\Lambda, c_i\geq 0.
\end{align*}
This puts the Schr\"{o}dinger kernel \eqref{KernelComponent} in the form of an exponential sum. We now review parts of Section 7 in \cite{Zha20} necessary for the estimate of this exponential sum.

\begin{defn}\label{definitionofarationallattice}
Let $L=\mathbb{Z}w_1+\cdots+\mathbb{Z}w_r$ be a lattice on an inner product space. We say $L$ is a \textit{rational lattice} provided that there exists some $D>0$ such that $\langle w_i,w_j\rangle\in D^{-1}\mathbb{Z}$. We call the number $D$ a \textit{period} of $L$. 
\end{defn}

From the theory of root systems, the weight lattice $\Lambda$ of $U/K$ is a rational lattice with respect to the Killing form. As a sublattice of $\Lambda$, the restricted root lattice $\Gamma$ is also rational. 

Let $f$ be a function on $\mathbb{Z}^r$ and define the \textit{difference operator} $D_i$'s by 
\begin{equation}\label{DefinitionDi}
D_{i}f(n_1,\cdots,n_r):=f(n_1,\cdots,n_{i-1}, n_i+1, n_{i+1},\cdots, n_r)-f(n_1,\cdots, n_r)
\end{equation} 
for $i=1,\cdots, r$. The Leibniz rule for $D_i$ reads
\begin{align}\label{Leibniz}
D_i(\prod_{j=1}^nf_j)=
\sum_{l=1}^n\sum_{1\leq k_1<\cdots<k_l\leq n}
D_if_{k_1}\cdots D_if_{k_l}\cdot \prod_{\substack{j \neq k_1,\cdots, k_l\\ 1\leq j\leq n}}f_j.
\end{align}
Note that there are $2^n-1$ terms in the right side of the above formula. 

We will need the following variant of Lemma 7.4 in \cite{Zha20}. 

\begin{lem}\label{WeylSumCor}
Let $L=\mathbb{Z}w_1+\cdots+\mathbb{Z}w_r$ be a rational lattice in the inner product space $(V,\langle\ ,\  \rangle)$ with a period $D$. Let $L^+=\mathbb{Z}_{\geq 0}w_1+\cdots+\mathbb{Z}_{\geq 0}w_r$. 
Let $\varphi$ be a bump function on $\mathbb{R}$ and $N\geq 1$. Let $f$ be a complex valued function on $L^+\cong(\mathbb{Z}_{\geq 0})^r$ such that there exists a constant $A\in\R$ for which 
\begin{align}\label{DecCon1}
|D_1^{\varepsilon_1}\cdots D_r^{\varepsilon_r}f(n_1,\ldots,n_r)|\lesssim N^{A-\varepsilon_1-\cdots-\varepsilon_r}
\end{align} 
for all $(\varepsilon_1,\ldots,\varepsilon_r)\in \{0,1\}^r$, uniformly for $0\leq n_i\lesssim N$, $i=1,\ldots, r$. 
Let
\begin{align}\label{absorb}
F(t,H)=\sum_{\lambda\in L^+}e^{-it|\lambda|^2+i\langle\lambda, H\rangle}\varphi\left(\frac{|\lambda|^2}{N^2}\right)\cdot f
\end{align}
for $t\in\mathbb{R}$ and $H\in V$. 
Then for $\frac{t}{2\pi D}\in\mathcal{M}_{a,q}$, we have
\begin{align}\label{appendix2}
|F(t,H)|\lesssim_{\varepsilon>0} \frac{N^{A+r+\varepsilon}}{[\sqrt{q}(1+N\|\frac{t}{2\pi D}-\frac{a}{q}\|^{1/2})]^{r}}
\end{align}
uniformly in $H\in V$.  
\end{lem}

\begin{proof}
By Weyl's differencing technique, 
\begin{align*}
|F(t,H)|^2&=\sum_{\lambda_1\in L^+,\lambda_2\in L^+}e^{-it(|\lambda_1|^2-|\lambda_2|^2)+i\langle\lambda_1-\lambda_2,H\rangle}\varphi\left(\frac{|\lambda_1|^2}{N^2}\right)\varphi\left(\frac{|\lambda_2|^2}{N^2}\right) f(\lambda_1)\overline{f(\lambda_2)}\\
&=\sum_{\substack{\mu\in L^+ \\ (\mu=\lambda_1+\lambda_2)}}e^{it|\mu|^2-i\langle\mu,H\r}\sum_{\substack{\lambda\in L^+\cap (\mu-L^+)\\ (\lambda=\lambda_1)}}e^{2i[\l\lambda,H\r-t\l\lambda,\mu\r]}\varphi\left(\frac{|\mu|^2}{N^2}\right)\varphi\left(\frac{|\mu-\lambda|^2}{N^2}\right)f(\mu)\overline{f(\mu-\lambda)}\\
&\lesssim\sum_{\substack{\mu\in L^+ \\ (\mu=\lambda_1+\lambda_2)}}\left|\sum_{\substack{\lambda\in L^+\cap (\mu-L^+)\\ (\lambda=\lambda_1)}}e^{2i[\l\lambda,H\r-t\l\lambda,\mu\r]}\varphi\left(\frac{|\mu|^2}{N^2}\right)\varphi\left(\frac{|\mu-\lambda|^2}{N^2}\right)f(\mu)\overline{f(\mu-\lambda)}\right|. \numberthis \label{sum}
\end{align*}
For $\mu\in L^+$, write 
\begin{align*}
\mu=n_1^\mu w_1+\cdots+n_r^\mu w_r. 
\end{align*}
Then  
\begin{align*}
\lambda\in\mu^+\cap(\mu-L^+) \t{ if and only if }\lambda=n_1w_1+\cdots+n_rw_r, \ 0\leq n_j\leq n_j^\mu, \ j=1,\ldots, r. 
\end{align*}
For $\lambda=n_1w_1+\cdots+n_rw_r$, let 
\begin{align*}
g(n_1,\ldots,n_r)=g(\lambda)=g(\lambda,N,\mu):=\phi\left(\frac{|\lambda|^2}{N^2}\right)\phi\left(\frac{|\mu-\lambda|^2}{N^2}\right)f_{\lambda}f_{\mu-\lambda}.
\end{align*}
By the assumption on $f$ and the Leibiniz rule for difference operators, we have 
\begin{align}\label{g}
|D_{1}^{\varepsilon_1}\cdots D_{r}^{\varepsilon_r}g(n_1,\cdots,n_r)|\lesssim N^{2A-\varepsilon_1-\cdots-\varepsilon_r}
\end{align} 
for all $(\varepsilon_1,\ldots,\varepsilon_r)\in\{0,1\}^r$, uniformly for $0\leq n_i\lesssim N$, $i=1,\ldots, r$. 
Now let $F^{\mu}=F^\mu(t,H)$ be the sum in \eqref{sum} inside of the absolute value. Then 
\begin{align*}
F^{\mu}=\sum_{0\leq n_r\leq n^\mu_r}e^{in_r\theta_r}\cdots\sum_{0\leq n_1\leq n^\mu_1} e^{in_1\theta_1} g(\lambda) 
\end{align*}
where 
\begin{align*}
\theta_j=\theta_{j}(t,H,\mu):=2[\l w_j,H\r-t\l w_j,\mu\r], \ j=1,\ldots,r.
\end{align*}
We can perform summation by parts on $F^\mu$ with respect to the variable $n_1$
\begin{align*}
\sum_{0\leq n_1\leq n^\mu_1} e^{in_1\theta_1} g(\lambda)
&=\frac{1}{1-e^{i\theta_1}}\sum_{0\leq n_1\leq n_1^\mu}e^{i(n_1+1)\theta_1}D_1g(\lambda)\\
&+\frac{1}{1-e^{i\theta_1}}g(0,n_2,\ldots,n_r)-\frac{e^{i(n_1^\mu+1)\theta_1}}{1-e^{i\theta_1}}g(n_1^\mu+1,n_2,\ldots,n_r).
\end{align*}
Then we can perform summation by parts with respect to other variables $n_2,\ldots,n_r$. But we require that only when 
\begin{align*}
|1-e^{i\theta_j}|\geq  N^{-1}, 
\end{align*}
do we carry out the procedure to the variable $n_j$. Using \eqref{g}, what we we end up with is an estimate 
\begin{align*}
|F^\mu|^2&\lesssim N^{2A}\prod_{j=1}^r\frac{1}{\max\{\frac{1}{N},|1-e^{i\theta_j}|\}}\\
&\lesssim N^{2A}\prod_{j=1}^r\frac{1}{\max\{\frac{1}{N},\|\frac{\theta_j}{2\pi}\|\}}\\
&\lesssim N^{2A}\prod_{j=1}^r\frac{1}{\max\{\frac{1}{N},\|\frac{w_j(H)}{\pi}-\frac{t\l w_j,\mu\r}{\pi}\|\}}.
\end{align*}
Since $D$ is a period of the lattice $L$, 
$
-2\l w_j,\mu\r\in D^{-1}\Z, \ \forall\mu\in L,\ j=1,\ldots,r$.  
Let 
\begin{align*}
m_j=-2\l w_j,\mu\r\cdot D, \ j=1,\ldots,r.
\end{align*} 
Since the map $\Lambda\ni\mu\mapsto (m_1,\ldots,m_r)\in \Z^r$ is one-one, we can write \eqref{sum} into 
\begin{align*}
|F(t,H)|^2&\lesssim N^{2A}\sum_{|m_1|\lesssim N,\ldots, |m_j|\lesssim N}\prod_{j=1}^r\frac{1}{\max\{\frac{1}{N},\|m_j\frac{t}{2\pi D}+\frac{w_j(H)}{\pi}\|\}}\\
&\lesssim N^{2A}\prod_{j=1}^r\left(\sum_{|m_j|\lesssim N}\frac{1}{\max\{\frac{1}{N},\|m_j\frac{t}{2\pi D}+\frac{w_j(H)}{\pi}\|\}}\right). 
\end{align*}
By a standard estimate as in deriving the classical Weyl inequality in one dimension (see Lemma 3.18 of \cite{Bou93}), we get 
\begin{align*}
\sum_{|m_j|\lesssim N}\frac{1}{\max\{\frac{1}{N},\|m_j\frac{t}{2\pi D}+\frac{w_j(H)}{\pi}\|\}}\lesssim\frac{N^2\log N}{[\sqrt{q}(1+N\|\frac{t}{2\pi D}-\frac{a}{q}\|^{1/2})]^2}
\end{align*}
for $\frac{t}{2\pi D}$ lying on the major arc $\mathcal{M}_{a,q}$. Hence
\begin{align*}
|F(t,H)|^2\lesssim \frac{N^{2A+2r}\log^r N}{[\sqrt{q}(1+N\|\frac{t}{2\pi D}-\frac{a}{q}\|^{1/2})]^{2r}}.
\end{align*}
\end{proof}

\begin{rem}\label{WeylSumRemark}
Let $\lambda_0$ be any constant vector in $V$ and $C$ any constant real number. If we slightly generalize the form of the function $F(t,H)$ in Lemma \ref{WeylSumCor} into
\begin{equation*}
F(t,H)=\sum_{\lambda\in L^+}e^{-it|\lambda+\lambda_0|^2+i\langle\lambda, H\rangle}\varphi\left(\frac{|\lambda+\lambda_0|^2+C}{N^2}\right)\cdot f
\end{equation*} 
then the proof still works. 
\end{rem}

We now play out our first application of Lemma \ref{WeylSumCor}. Let $e$ denote the identity element of $U$ and specialize the Schr\"{o}dinger kernel \eqref{KernelComponent} to the identity coset $x=eK$. Noting that $\Phi_\lambda(eK)=1$, we have 
\begin{align}\label{x=K}
K_N(t,eK)=\sum_{\lambda\in\Lambda^+}\varphi\left(\frac{-|\lambda+\rho|^2+|\rho|^2}{N^2}\right)e^{it(-|\lambda+\rho|^2+|\rho|^2)} d_\lambda.
\end{align} 
Recall that $K_N(t,\cdot)$ is periodic in $t$ and let $T$ be a period. 

\begin{prop}\label{Dispersivex=K}
For all globally symmetric spaces $U/K$ of compact type, we have 
\begin{align*}
|K_N(t,eK)|\lesssim_{\varepsilon>0} \frac{N^{d+\varepsilon}}{[\sqrt{q}(1+N\|\frac{t}{T}-\frac{a}{q}\|^{1/2})]^{r}}
\end{align*}
for $\frac{t}{T}\in\mathcal{M}_{a,q}$.
\end{prop}
\begin{proof}
$d_\lambda$ is a polynomial in $\lambda\in\Lambda$ of degree $d-r$ by Lemma \ref{dimensionformula}. Thus $d_\lambda$ as a function on $\Lambda^+\cong(\mathbb{Z}_{\geq 0})^r$ satisfies \eqref{DecCon1} with $A=d-r$. Then this is a direct consequence of  Lemma \ref{WeylSumCor}. 
\end{proof}

We now strengthen Proposition \ref{Dispersivex=K}.  
\begin{prop}\label{DispersiveNeighborhood}
Let $M$ be a globally symmetric space of the compact type. Let $d(\cdot,\cdot)$ be the Riemannian distance function on $M$ induced from the Killing form. 
Then we have 
\begin{align*}
|K_N(t,x)|\lesssim_{\varepsilon>0} \frac{N^{d+\varepsilon}}{(\sqrt{q}(1+N\|\frac{t}{T}-\frac{a}{q}\|^{1/2}))^{r}}
\end{align*}
for $\frac{t}{T}\in\mathcal{M}_{a,q}$, uniformly for $d(x,eK)\lesssim N^{-1}$. 
\end{prop}

The proof of this proposition hinges on an integral representation of spherical functions in a neighborhood of $o$. We continue the notations in the Preliminaries. Let $\mathfrak{n}^{\mathbb{C}}, \mathfrak{a}^{\mathbb{C}}, \mathfrak{k}^{\mathbb{C}}$ be respectively the complexification of $\mathfrak{n}, \mathfrak{a}, \mathfrak{k}$ in $\mathfrak{g}^\C$. Let $G^\C$ be a simply connected group whose Lie algebra is $\mathfrak{g}^\C$ so that $U$ becomes an analytic subgroup of $G^\C$.  
By Section 9.2 Ch. III in \cite{Hel08}, the mapping 
\begin{align*}
(X,H,T)\mapsto \exp X\exp H\exp T, \ X\in\mathfrak{n}^{\mathbb{C}}, H\in\mathfrak{a}^{\mathbb{C}}, T\in\mathfrak{k}^{\mathbb{C}}
\end{align*}
is a holomorphic diffeomorphism of a neighborhood $\mathcal{U}^{\mathbb{C}}$ of $G^{\mathbb{C}}$ such that $\mathcal{U}=\mathcal{U}^{\mathbb{C}}\cap U$ is invariant under the maps $u\mapsto kuk^{-1}$, $k\in K$. This induces the map 
\begin{align*}
{\bf A}: \exp X\exp H\exp T\to H
\end{align*}
that sends $\mathcal{U}^{\mathbb{C}}$ into $\mathfrak{a}^{\mathbb{C}}$. Let $\Phi_\lambda$ be the spherical function associated to $\lambda\in \Lambda^+$. By Lemma 9.2 of Ch. III in \cite{Hel08}, 
\begin{align}\label{PhiIntegral}
\Phi_\lambda(uK)=\int_K e^{-\lambda({\bf A}(ku^{-1}k^{-1}))}\ dk, \ u\in \mathcal{U}.
\end{align}
Note that the map $u\mapsto kuk^{-1}$ preserves the distance $d_U(\cdot,e)$ to the identity $e$ of $U$. Let $N\geq1$ be large enough so that $\{u\in U: d_U(u,e)\lesssim N^{-1}\}\subset \mathcal{U}$. Then 
\begin{align}\label{|Akuk-1|}
|{\bf A}(ku^{-1}k^{-1})|\lesssim N^{-1}
\end{align}
uniformly for $d_U(u,e)\lesssim N^{-1}$ and $k\in K$. Here the norm on $\mathfrak{a}^{\mathbb{C}}$ of course comes from the Killing form. Write $\lambda=n_1w_1+\cdots+n_rw_r$, $n_i\in\mathbb{Z}_{\geq 0}$, viewing $\Phi_\lambda(uK)=\Phi(\lambda,uK)$ as a function of $\lambda\in(\mathbb{Z}_{\geq 0})^r$, \eqref{PhiIntegral} and \eqref{|Akuk-1|} imply that $\Phi(\lambda,uK)$ satisfies an inequality of the type \eqref{DecCon1} as follows. 

\begin{lem}\label{differencePhi}
\begin{align*}
|D_{i_1}\cdots D_{i_n}\Phi(n_1,\cdots,n_r, uK)|\lesssim N^{-n}
\end{align*} 
holds uniformly for $0\leq n_i\lesssim N$ and $d_U(u,e)\lesssim N^{-1}$, for all $i_j=1,\cdots, r$ and $n\in\mathbb{Z}_{\geq0}$. 
\end{lem}

\begin{proof}[Proof of Proposition \ref{DispersiveNeighborhood}]
For each $u\in U$, there exists $k\in K$ such that 
$d(uK,eK)=d_U(uk,e)$ by definition of the Riemannian metrics. 
Suppose $u\in U$ such that $d(uK,eK)\lesssim N^{-1}$. Then there exists $k_u\in K$ such that $d_U(uk_u,e)=d(uK,eK)\lesssim N^{-1}$. 
Write $\lambda=n_1w_1+\cdots+n_rw_r$, and 
\begin{align*}
\Phi(n_1,\ldots,n_r,uK)=\Phi(\lambda,uK)=\Phi_\lambda(uK)=\int_K e^{-\lambda({\bf A}(k(uk_u)^{-1}k^{-1}))}\ dk, \ u\in \mathcal{U}.
\end{align*}
By Lemma \ref{differencePhi}, we have 
\begin{align*}
|D_{i_1}\cdots D_{i_n}\Phi(n_1,\ldots,n_r, uK)|\lesssim N^{-n}
\end{align*} 
holds uniformly for $0\leq n_i\lesssim N$ and  $d(uK,eK)\lesssim N^{-1}$, for all $i_j=1,\cdots, r$ and $n\in\mathbb{Z}_{\geq0}$. Using the fact that $d_\lambda$ is a polynomial in $\lambda$ of degree $d-r$ and applying the Leibniz rule \eqref{Leibniz}, we have that $f(\lambda):=d_\lambda\Phi_\lambda(uK)$ as a function of $\lambda$ satisfies \eqref{DecCon1} with $A=d-r$. Then we apply Lemma \ref{WeylSumCor} to finish the proof.  
\end{proof}

Finally, we upgrade Proposition \ref{DispersiveNeighborhood} to cover the cases when $[iH]\in A$ is within the distance of about $N^{-1}$ from any corner.

\begin{prop}\label{DispersiveCorner}
Let $[iH_0]\in A$ be any corner. Then it holds that 
\begin{align}\label{DispersiveNeighborhood(i)}
|K_N(t,[iH])|\lesssim_{\varepsilon>0} \frac{N^{d+\varepsilon}}{(\sqrt{q}(1+N\|\frac{t}{T}-\frac{a}{q}\|^{1/2}))^{r}}
\end{align}
\eqref{DispersiveNeighborhood(i)} holds uniformly for 
$\frac{t}{T}\in \mathcal{M}_{a,q}$ and $[iH]\in A$ such that $d([iH], [iH_0])\lesssim N^{-1}$.  
\end{prop}

To prove this proposition, we describe another important characterization of spherical functions. For $\mu,\lambda\in \Lambda$, let $\mu\leq \lambda$ denote the statement that $\lambda-\mu\in 2\mathbb{Z}_{\geq 0}\alpha_1+\cdots+2\mathbb{Z}_{\geq 0}\alpha_r$. 
For $\mu\in\Lambda^+$, define
\begin{align*}
M(\mu)=\sum_{s\in W} e^{s\mu}. 
\end{align*}
Then define the Heckman-Opdam polynomials (the generalized Jacobi polynomials) $P(\lambda)$, $\lambda\in \Lambda^+$, by 
\begin{align*}
P(\lambda)=\sum_{\mu\in \Lambda^+,\mu\leq\lambda}c_{\lambda ,\mu}M(\mu), \ c_{\lambda, \lambda}=1
\end{align*} 
such that 
\begin{align*}
\int_A P(\lambda)\cdot\overline{M(\mu)}\cdot\left|\prod_{\alpha\in\Sigma^+}(e^{\alpha}-e^{-\alpha})^{m_\alpha}\right|\ da=0, \ \t{for any } \mu\in\Lambda^+\t{ such that } \mu<\lambda. 
\end{align*}
Here $da$ is the normalized Haar measure on the $A$. Normalize $P(\lambda)$ by 
\begin{align*}
\tilde{P}(\lambda)=\frac{P(\lambda)}{P(\lambda)(o)}.
\end{align*}

\begin{thm}[Corollary 5.2.3 in Part I of \cite{HS94}]\label{SphericalHO}
The spherical functions on $U/K$ restricted on $A$ are given by the normalized Heckman-Opdam polynomials: 
\begin{align*}
\Phi_\lambda=\tilde{P}(\lambda), \ \lambda\in \Lambda^+. 
\end{align*}
\end{thm}

\begin{cor}\label{cornerto0}
Let $[iH_0]\in A$ be a corner. Then 
\begin{align*}
\Phi_\lambda([iH+iH_0])=e^{i\lambda(H_0)}\Phi_\lambda([iH]), \ H\in \mathfrak{a}, \ \lambda\in\Lambda^+.
\end{align*}
\end{cor}
\begin{proof}
By the above theorem and the definition of Heckman-Opdam polynomials, it suffices to show that for any $\lambda,\mu\in\Lambda^+$ such that $\mu\leq\lambda$ and $s\in W$, 
\begin{align*}
e^{i(s\mu)(H_0)}=e^{i\lambda(H_0)}. 
\end{align*}
This is reduced to showing $(s\mu-\lambda)(H_0)\in 2\pi\mathbb{Z}$, and by the definition of $[iH_0]$ as a corner, it is further reduced to $s\mu-\lambda\in 2\mathbb{Z}\alpha_1+\cdots+2\mathbb{Z}\alpha_r$. By the fact $\mu\leq\lambda$, it then suffices to show 
$s\mu-\mu\in 2\mathbb{Z}\alpha_1+\cdots+2\mathbb{Z}\alpha_r$ for any $\mu\in \Lambda$ and $s\in W$. But this is a standard fact of root system theory (see Corollary 4.13.3 in \cite{Var84}).  
\end{proof}

Let $\Gamma=2\mathbb{Z}\alpha_1+\cdots+2\mathbb{Z}\alpha_r$. 
The above corollary implies that for $\lambda\in \Gamma$ and $\mu\in\Lambda^+$ such that $\lambda+\mu\in \Lambda^+$, 
\begin{align}\label{Philambda+mu}
\Phi_{\lambda+\mu}([iH+iH_0])=e^{i\mu(H_0)}\Phi_{\lambda+\mu}([iH]).
\end{align} 
This inspires a decomposition of $\Lambda^+$ and thus of the Schr\"odinger kernel \eqref{KernelComponent}, so to make applicable the techniques in proving Theorem \ref{DispersiveNeighborhood} for the proof of Theorem \ref{DispersiveCorner}. 

\begin{proof}[Proof of Proposition \ref{DispersiveCorner}]
First note that all the fundamental weights $w_1,\cdots, w_r$ are rational linear combinations of the roots. Thus there exists some $B\in\mathbb{N}$ such that $Bw_i\in \Gamma$ for all $i$. Define  
\begin{align*}
\Gamma^+_1=\mathbb{Z}_{\geq 0}Bw_1+\cdots+\mathbb{Z}_{\geq 0}Bw_r. 
\end{align*}
Then $\Lambda^+/\Gamma^+_1=\{n_1w_1+\cdots+n_rw_r: n_i=0,\cdots, B-1, i=1,\cdots, r\}$ and we decompose 
\begin{align*}
\Lambda^+=\bigsqcup_{\mu\in \Lambda^+/\Gamma^+_1}(\Gamma^+_1+\mu). 
\end{align*}
This yields a decomposition of the Schr\"odinger kernel
\begin{align*}
K_N&=\sum_{\mu\in \Lambda^+/\Gamma^+_1}K_{N}^\mu,\numberthis \label{WeighttoRoot}\\ 
K_{N}^\mu&=\sum_{\lambda\in\Gamma^+_1}\varphi\left(\frac{-|\lambda+\mu+\rho|^2+|\rho|^2}{N^2}\right)e^{it(-|\lambda+\mu+\rho|^2+|\rho|^2)} d_{\lambda+\mu} \Phi_{\lambda+\mu}. 
\end{align*}
By the finiteness of $\Lambda^+/\Gamma^+_1$, it suffices to prove 
\eqref{DispersiveNeighborhood(i)} replacing $K_N$ by $K_N^\mu$.  By \eqref{Philambda+mu}, 
\begin{align*}
K_{N}^\mu(t,[iH+iH_0])=e^{i\mu(H_0)}\sum_{\lambda\in\Gamma^+_1}\varphi\left(\frac{-|\lambda+\mu+\rho|^2+|\rho|^2}{N^2}\right)e^{it(-|\lambda+\mu+\rho|^2+|\rho|^2)} d_{\lambda+\mu} \Phi_{\lambda+\mu}([iH]).
\end{align*}
The assumption $d([iH+iH_0], [iH_0])\lesssim N^{-1}$ yields $d([iH],eK)\lesssim N^{-1}$. 
The rest of the proof now follows exactly as that of Proposition \ref{DispersiveNeighborhood}. Write 
$\lambda=n_1Bw_1+\cdots+n_rBw_r$, and 
$\Phi(n_1,\ldots,n_r,uK):=\Phi_{\lambda+\mu}(uK)$. Then the analogue of Lemma \ref{differencePhi} for $\Phi(n_1,\ldots,n_r,uK)$ still holds following the same proof, and we apply Lemma \ref{WeylSumCor} to $f(\lambda)=d_{\lambda+\mu}\Phi_{\lambda+\mu}([iH])$ and finish the proof. 
\end{proof}

We can now finish the 
\begin{proof}[Proof of Proposition \ref{Lppart1}]
Recall that corners $[iH_0]$ are such that $\alpha(H_0)\in \pi \mathbb{Z}$ for all $\alpha\in\Sigma$. Consider $[iH]\in A$. As the Riemannian distance $d(\cdot,\cdot)$ is induced from the Killing form, $d([iH],[iH_0])\leq N^{-1}$ implies that the  distance between $H-H_0$ and the lattice $2\pi  \mathbb{Z}\frac{H_{\alpha_1}}{\langle\alpha_1,\alpha_1\rangle}+\cdots+2\pi  \mathbb{Z}\frac{H_{\alpha_r}}{\langle\alpha_r,\alpha_r\rangle}$ is no larger than $N^{-1}$. In particular, this implies that for all $\alpha\in \Sigma$, $|\sin\langle\alpha,H\rangle |\lesssim N^{-1}$, and so further 
$$D([iH])=\prod_{\alpha\in \Sigma^+}|\sin\langle\alpha,H\rangle |^{m_\alpha}\lesssim N^{-\sum_{\alpha\in\Sigma^+} m_\alpha}=N^{-(d-r)}.$$
By Proposition \ref{DispersiveCorner} and an application of \eqref{integralformula}, then we have 
\begin{align*}
\|K_N(t,\cdot )\|_{L^p\left(U_{[iH_0],N^{-1}}\right)}
&=c\left(\int_{|H-H_0|\leq N^{-1}}|K_N(t,[iH])|^p D([iH]) \ dH\right)^{1/p}\\
&\lesssim_\varepsilon \frac{N^{d-\frac{d-r}{p}+\varepsilon}}{[\sqrt{q}(1+N\|\frac{t}{T}-\frac{a}{q}\|^{1/2})]^r}\cdot\left(\int_{|H-H_0|\leq N^{-1}}\ dH\right)^{1/p}\\
&\lesssim_\varepsilon \frac{N^{d-\frac{d-r}{p}-\frac{r}{p}+\varepsilon}}{[\sqrt{q}(1+N\|\frac{t}{T}-\frac{a}{q}\|^{1/2})]^r}= \frac{N^{d-\frac{d}{p}+\varepsilon}}{[\sqrt{q}(1+N\|\frac{t}{T}-\frac{a}{q}\|^{1/2})]^r}.
\end{align*}
\end{proof}

\section{Dispersive estimates away from the corners}

We have in Proposition \ref{Lppart1} derived favorable estimates for the Schr\"odinger kernel on $N^{-1}$ neighborhoods of the corners for a general compact globally symmetric space. For estimates away from the corners, in \cite{Zha21} we established the desired $L^\infty$ estimates using Clerc's formula for spherical functions on a general compact symmetric space. However, $L^p$ estimates of the Schr\"odinger kernel as in Theorem \ref{keyproposition} are still missing in the literature for general compact symmetric spaces. We now provide these estimates for products of odd-dimensional spheres, and the major tools to use are the explicit formulas of the ultraspherical polynomials which express the spherical functions on odd-dimensional spheres. 

\begin{proof}[Proof of Theorem \ref{keyproposition}] 
We first treat the case when $M=\mathbb{S}^d=\text{SO}(d+1)/\text{SO}(d)$ is a single sphere of dimension $d=2\lambda+1$. As reviewed in the Preliminaries, there are two corners on $M$, namely the two poles $p_1$ and $p_2$ which correspond to $0$ and $\pi$ on any maximal torus $\mathbb{R}/2\pi\mathbb{Z}$. 
Let $U(p_i,N^{-1})$ be the collection of points of $M$ which are within a distance of $N^{-1}$ from $p_i$, $i=1,2$. 
Proposition \ref{Lppart1} already provides the following estimate 
$$\|K_N(t,\cdot)\|_{L^p\left(U(p_i,N^{-1})\right)}\lesssim_\varepsilon \frac{N^{d-\frac{d}{p}+\varepsilon}}{[\sqrt{q}(1+N\|\frac{t}{T}-\frac{a}{q}\|^{1/2})]^r}
$$
for $\frac{t}{T}\in\mathcal{M}_{a,q}$ and $p>0$, $i=1,2$. It suffices to get the desired estimate for 
$\|K_N(t,\cdot)\|_{L^p\left(\mathbb{S}^d\setminus \left[U(p_1,N^{-1})\cup U(p_2,N^{-1})\right]\right)}$. Note that $\mathbb{S}^d\setminus \left[U(p_1,N^{-1})\cup U(p_2,N^{-1})\right]$ corresponds to $(N^{-1},\pi-N^{-1})\cup (\pi+N^{-1},2\pi-N^{-1})$ in any maximal torus $A=\mathbb{R}/2\pi\mathbb{Z}$. 
Using the integral formula \eqref{integralformula} which now reads 
$$\int_{U/K}f\ d(uK) =\int_{0}^{2\pi} f(\theta) |\sin\theta|^{d-1} \ d\theta$$
for any continuous $K$-invariant function $f$ on $U/K$, we have 
\begin{align}\label{integralformulaaway}
\|K_N(t,\cdot)\|_{L^p\left(\mathbb{S}^d\setminus  \left[U(p_1,N^{-1})\cup U(p_2,N^{-1})\right]\right)}
=\left(\left(\int_{N^{-1}}^{\pi-N^{-1}}+\int_{\pi+N^{-1}}^{2\pi-N^{-1}}\right)|K_N(t,\theta)|^p|\sin\theta|^{d-1}\ d\theta \right)^{1/p}.
\end{align}
As an $\text{SO}(d)$-invariant function on $\text{SO}(d+1)/\text{SO}(d)$, 
$K_N(t,\cdot)$ is invariant under the Weyl group action $\theta\mapsto 2\pi-\theta$, thus it suffices to estimate 
the  integral $\int_{N^{-1}}^{\pi-N^{-1}}$ in the above. 
The Schr\"odinger kernel reads 
\begin{align*}
K_N(t,\theta)=\sum_{n\in\mathbb{Z}_{\geq 0}} \varphi\left(\frac{(n+\lambda)^2-\lambda^2}{N^2}\right)e^{-it[(n+\lambda)^2-\lambda^2]}d_n\Phi_n^{(\lambda)}(\theta).  
\end{align*}
By \eqref{sphericalsphere}, we have 
$$K_N(t,\theta)=\sum_{\nu=0}^{\lambda-1}K_{N}^{(\nu)}(t,\theta),$$
where 
\begin{align*}
K_{N}^{(\nu)}(t,\theta)=\frac{2}{(2\sin\theta)^{\nu+\lambda}}\sum_{n\in\mathbb{Z}_{\geq 0}} \varphi\left(\frac{(n+\lambda)^2-\lambda^2}{N^2}\right)e^{-it[(n+\lambda)^2-\lambda^2]}d_nC_{n,\nu}\cos((n-\nu+\lambda)\theta-(\nu+\lambda)\pi/2),
\end{align*}
with 
\begin{align*}
C_{n,\nu}=\binom{n+2\lambda-1}{n}^{-1}\binom{n+\lambda-1}{n}\binom{\nu+\lambda-1}{\nu}\frac{(1-\lambda)\cdots(\nu-\lambda)}{(n+\lambda-1)\cdots(n+\lambda-\nu)}. 
\end{align*}
\begin{lem}
For $\nu=0,\ldots,\lambda-1$, let 
$$\kappa_N^{(\nu)}(t,\theta)=\sum_{n\in\mathbb{Z}_{\geq 0}} \varphi\left(\frac{(n+\lambda)^2-\lambda^2}{N^2}\right)e^{-it[(n+\lambda)^2-\lambda^2]}d_nC_{n,\nu}\cos((n-\nu+\lambda)\theta-(\nu+\lambda)\pi/2).$$
Then 
$$\|\kappa_N^{(\nu)}(t,\cdot)\|_{L^\infty(A)} \lesssim_{\varepsilon} \frac{N^{\lambda-\nu+1+\varepsilon}}{\sqrt{q}(1+N\|\frac{t}{T}-\frac{a}{q}\|^{1/2})}$$
for $\frac{t}{T}\in\mathcal{M}_{a,q}$. 
\end{lem}
\begin{proof}
Note that $d_n$ is polynomial in $n$ of degree $d-1=2\lambda$, so we can write $d_nC_{n,v}=\frac{f(n)}{g(n)}$ such that $f(n)$ and $g(n)$ are polynomials of degree $3\lambda-1$ and $2\lambda-1+\nu$ respectively. As $\nu\leq \lambda-1$, it holds  $2\lambda-1+\nu\leq 3\lambda-2$. This implies that $d_nC_{n,v}$ satisfies an estimate of the form \eqref{DecCon1} as follows
\begin{align*}
|D^{\mu}(d_nC_{n,\nu})|\lesssim N^{\lambda-\nu-\mu}
\end{align*}
for $\mu=0,1$, 
uniformly for $0\leq n\lesssim N$. Using 
$\cos \phi=\frac{1}{2}(e^{i\phi}+e^{-i\phi})$, then we apply Lemma \ref{WeylSumCor} to finish the proof. 
\end{proof}

Using this lemma, we now estimate 
\begin{align*}
\int_{N^{-1}}^{\pi-N^{-1}}|K_N^{(\nu)}(t,\theta)|^p|\sin\theta|^{d-1}\ d\theta
&\lesssim \|\kappa_N^{(\nu)}(t,\cdot)\|^p_{L^\infty(A)}
\int_{N^{-1}}^{\pi-N^{-1}} |\sin\theta|^{d-1-p(\nu+\lambda)}\ d\theta.
\end{align*}
We require 
$$d-1-p(\nu+\lambda)\leq -1\text{ for any }\nu=0,\ldots,\lambda-1,$$
which is equivalent to 
$$p\geq\frac{d}{\lambda}=\frac{2d}{d-1}.$$
Under this requirement of $p$, we have 
$$\int_{N^{-1}}^{\pi-N^{-1}} |\sin\theta|^{d-1-p(\nu+\lambda)}\ d\theta\lesssim_\varepsilon N^{p(\nu+\lambda)-d+\varepsilon}.$$
Then
\begin{align*}
\int_{N^{-1}}^{\pi-N^{-1}}|K_N(t,\theta)|^p|\sin\theta|^{d-1}\ d\theta 
 \lesssim_{\varepsilon} \left(\frac{N^{\lambda-\nu+1+\varepsilon}}{\sqrt{q}(1+N\|\frac{t}{T}-\frac{a}{q}\|^{1/2})}\right)^p
N^{p(\nu+\lambda)-d+\varepsilon}
\end{align*}
which by \eqref{integralformulaaway} yields the desired estimate 
\begin{align}\label{boundforsinglesphere}
\|K_N(t,\cdot)\|_{L^p\left(\mathbb{S}^d\setminus \left[U(p_1,N^{-1})\cup U(p_2,N^{-1})\right]\right)}
\lesssim_{\varepsilon}\frac{N^{d-\frac{d}{p}+\varepsilon}}{\sqrt{q}(1+N\|\frac{t}{T}-\frac{a}{q}\|^{1/2})}\text{ for all }p\geq \frac{2d}{d-1}.
\end{align}
We have finish the proof for single spheres. 
For the cases when $M=\mathbb{S}^{d_1}\times\cdots\times\mathbb{S}^{d_r}$ is a product of odd-dimensional spheres, 
the mollified Schr\"odinger kernel $K_N(t,\cdot)$ for $M$ is the product of the mollified Schr\"odinger kernels $K_{N}^{(j)}(t,\cdot)$ for each of the $\mathbb{S}^{d_j}$'s ($j=1,\ldots,r$). Then the $L^p(M)$ norm of $K_N(t,\cdot)$ is simply the product of the $L^p(\mathbb{S}^{d_j})$ norm of $K_{N}^{(j)}(t,\cdot)$, each bounded by \eqref{boundforsinglesphere}. Multiplying the bounds together over $j$, we finish the proof of Theorem \ref{keyproposition}. 
\end{proof}

\section*{Acknowledgments} The author thanks the referee for valuable feedback.

\end{document}